\documentclass[letterpaper, 10 pt, conference]{ieeeconf}  

\IEEEoverridecommandlockouts                              
\usepackage{times} 
\usepackage{amsmath,amsfonts} 
\usepackage{amssymb}  
\usepackage{tabularx}
\usepackage[normalem]{ulem}
\usepackage{dsfont}
\usepackage{graphicx}
\usepackage[dvipsnames,usenames]{color}
\usepackage[colorlinks,%
linkcolor=BrickRed,%
filecolor=BrickRed,%
citecolor=RoyalPurple,%
]{hyperref}
\usepackage{color}
\usepackage{todonotes}

\usepackage{subfigure}
\usepackage{float}

\newcommand{\calF}{{\mathcal F}}
\newcommand{\calT}{{\mathcal T}}

\def\Re{\mathbb{R}}

\def\argmin{\mathop{\text{\rm arg\,min}}}

\def\notes#1{\marginpar{\tiny #1}\typeout{Notes!
Notes!
Notes!
}}
\renewcommand{\notes}[1]{\typeout{notes!}}

\def\Re{\field{R}}

\def\Expect{{\mathbb E}}

\newtheorem{definition}{Definition}
\newtheorem{lemma}{Lemma}
\newtheorem{remark}{Remark}
\newtheorem{proposition}{Proposition}

\def\beq{\begin{eqnarray}} 
\def\bc{\begin{center}} 
\def\be{\begin{enumerate}}
\def\bi{\begin{itemize}} 
\def\bs{\begin{small}}
\def\bS{\begin{slide}}
\def\ec{\end{center}} 
\def\ee{\end{enumerate}}
\def\ei{\end{itemize}}
\def\es{\end{small}}
\def\eS{\end{slide}}
\def\eeq{\end{eqnarray}}

\newcommand{\newP}[1]{\noindent{\bf #1:}}

\def\Re{\mathbb{R}}

\def\argmin{\mathop{\text{\rm arg\,min}}}

\renewcommand{\Re}{\mathbb{R}}

\newcounter{rmnum}

\newcounter{anum}

\newcommand{\calS}{\mathcal{S}}
\newcommand{\calA}{\mathcal{A}}

\newcommand{\ALG}{{OTPF}}

\title{\LARGE \bf
	Optimal Transport Particle Filters 
}

\author{Mohammad Al-Jarrah$^\star$, Bamdad Hosseini$^\dagger$, Amirhossein Taghvaei$^\star$
    {\thanks{$^\star$Department of Aeronautics \& Astronautics, University of Washington, Seattle; {\tt\small mohd9485@uw.edu,amirtag@uw.edu}.}}
    {\thanks{$^\dagger$Department of Applied Mathematics, University of Washington, Seattle
        {\tt\small bamdadh@uw.edu}.}}
}

\begin{document}

      \maketitle
	 \thispagestyle{empty}
	 \pagestyle{empty}


	 \begin{abstract}
 This paper is concerned with the theoretical and computational development of a new class of nonlinear filtering algorithms called the optimal transport particle filters (OTPF). The algorithm is based on a recently introduced variational formulation of the Bayes' rule, which aims to find the Brenier optimal transport map between the prior and the posterior distributions as the solution to a stochastic optimization problem. On the theoretical side, the existing methods for the error analysis of particle filters and stability results for optimal transport map estimation are combined to obtain uniform error bounds for the filter's performance in terms of the optimization gap in solving the variational problem. The error analysis reveals a bias-variance trade-off that can ultimately be used to understand if/when the curse of dimensionality can be avoided in these filters. On the computational side, the proposed algorithm is evaluated on a nonlinear filtering example in comparison with the ensemble Kalman filter (EnKF) and the sequential importance resampling (SIR) particle filter.

	\end{abstract}

    \section{Introduction}
    Optimal transportation (OT) theory has gained significant interest recently because it {provides  natural geometrical
and mathematical tools for analysis and manipulation of probability distributions}~\cite{ambrosio2005gradient,villani2009optimal,santambrogio2015optimal}. 
In particular, two geometric
notions are of key importance: (i) a metric to measure the similarity/discrepancy between probability distributions, i.e., the Wasserstein metric, and (ii) a map to transport one distribution to the other, i.e., the OT {or Monge} map. In contrast to their information-theoretic counterparts {(such as the Kullback-Liebler (KL) divergence)},
these {OT metrics} respect the geometry of the data and are {often more} robust against perturbations {and errors}. Due to {these} unique features, {OT metrics and maps} have been successfully employed in a variety of applications, including {generative modeling and sampling}~\cite{arjovsky2017wasserstein,tolstikhin2017wasserstein}, domain adaptation~\cite{courty2015optimal}, and image processing~\cite{kolouri2017optimal,dominitz2010texture,rabin2011wasserstein,ferradans2014regularized,su2015optimal}.  

Given the significance of OT theory, there has been a growing line of research to apply {OT tools for nonlinear filtering and Bayesian inference}~\cite{el2012bayesian,reich13,reich2019data,taghvaei2021optimal}. The central idea {here}
is to view  
the filtering/Bayesian {update} step as the problem of transporting the prior distribution {(of the current state)} to the posterior distribution {(of the future state)}. This perspective {has} led to the development of new algorithms for Bayesian inference, namely learning triangular transport maps with polynomial, {radial basis, and neural net parameterizations}  to sample from {posteriors~\cite{el2012bayesian,marzouk2016introduction, kovachki2020conditional},}   
ensemble transform particle filters~\cite{reich2013nonparametric}, and OT interpretations of the feedback particle filter algorithm~\cite{yang2016,AmirACC2016,taghvaei2020optimal,taghvaei2021optimal}.

This paper builds on the authors' recent work~\cite{taghvaei2022optimal} where an OT-based variational formulation of the Bayes' law was introduced to learn the OT map from the prior to the posterior distribution for any value of the observation signal. 
In this formulation, the conditional distribution $P_{X|Y}$, of a hidden random variable $X$ given the observation $Y$,  is identified as
$P_{X|Y}(\cdot|y) =  \nabla_x \bar f(\cdot;y) \# P_X(\cdot), \: \forall y$, i.e.,
the push-forward of the prior distribution $P_X$ with respect to a  map of the form $\nabla \bar f$
where  $\bar f$ is a real-valued function that solves the optimization problem:
\begin{equation}
\begin{aligned}
\bar f &=  \argmin_{f\in \text{CVX}_x} \,\, \mathbb E[f( \bar X;Y)+f^\ast(X;Y)].
\end{aligned}
 \label{eq:var-form}
 \end{equation}
 Here $\bar X$ is an independent copy of $X$ and  $\text{CVX}_x$ denotes the set of functions $f(x;y)$ that are convex with respect to the $x$ argument for any fixed $y$, 
and $f^\ast$ is the convex conjugate of $f$ with respect  to the $x$ argument. 

The above variational formulation enjoys three key features that distinguish it from prior works: (i) it is simulation-based in the sense that it is possible to approximate the objective function in terms of samples
from the joint distribution $P_{XY}$ and does not require an explicit formula for the likelihood; (ii) The variational formulation enables new approximation methods for computing the posterior distribution by
choosing different subsets/parameterizations of the set $\text{CVX}_x$; (iii) The problem \eqref{eq:var-form} is stochastic and can be solved efficiently using recent machine learning techniques, for example, $f$ can be parameterized as a deep neural network and trained using stochastic gradient descent.
Problem \eqref{eq:var-form} can be obtained as the dual form of a block-triangular Monge problem 
between the independent coupling $P_X\otimes P_Y$ and the joint distribution~$P_{XY}$.
Similar variational formulations arise in block-triangular transport of distributions in the context of  conditional generative models; for example ~\cite{spantini2019coupling,kovachki2020conditional,ray2022efficacy,shi2022conditional,siahkoohi2021preconditioned}.

The objective of the current paper is to use the formulation~\eqref{eq:var-form} to develop a new nonlinear filtering algorithm, called {\it optimal transport particle filter} (OTPF),
and provide preliminary {theoretical analysis and numerical validation}  of the algorithm. The proposed
algorithm can be viewed as a nonlinear and non-Gaussian generalization of the ensemble Kalman filter algorithm (EnKF)~\cite{evensen2006,calvello2022ensemble} and the discrete-time counterpart of the feedback particle filter (FPF) algorithm~\cite{taoyang_TAC12,yang2016}
(OTPF~solves the gain function approximation and the numerical time discretization problems in the FPF altogether by solving the proposed variational problem~\eqref{eq:var-form}).

The theoretical analysis of the paper is concerned with the error analysis of the proposed algorithm.
In particular, we study how errors solving  problem \eqref{eq:var-form} at each time step affect the overall
  performance of the filtering algorithm. To do so, we adapt the existing methods for error analysis of particle filters (PF) to obtain a uniform bound on the filtering error in terms of the  approximation  error of the OT map~\cite{del2001stability,van2009uniform,cappe2009inference,delmoralbook}. These results are based on a strong notion of uniform geometric filter stability~\cite{atar1997exponential}, which is common in the analysis of PF. Next, we combine this with
stability results for the estimation of OT maps~\cite{hutter2019minimax} which relates the approximation error of the
map to the optimization gap of~\eqref{eq:var-form} (see Lemma~\ref{lem:opt-gap}) .
The error analysis is carried out for the mean-field limit of the algorithm and a variant of the particle system that involves an additional resampling step which makes the particles independent of each other and significantly simplifies
the  analysis.

 The numerical experiments  qualitatively and quantitatively evaluate the performance of the OTPF in comparison with the EnKF algorithm~\cite{evensen2006,calvello2022ensemble} and the sequential importance resampling (SIR) PF~\cite{doucet09}. In particular, we consider a linear stable   dynamical system with three different observation functions: linear, quadratic, and cubic. The numerical results illustrate the versatile nature of the OTPF 
compared to the other two methods.

The rest of the paper is organized as follows: Section~\ref{sec:filtering} reviews the filtering problem and
  equations, and introduces the notion of filter stability;
  Section~\ref{sec:OTPF} outlines the OTPF algorithm in detail; 
  Section~\ref{sec:error-analysis} presents the error analysis; and
  Section~\ref{sec:numerics} contains the numerical experiments.  

	\section{Problem Formulation}\label{sec:filtering}
        \subsection{Filtering problem}
Consider a discrete-time stochastic dynamic system given by the update equations
\begin{subequations}\label{eq:model}
\begin{align}\label{eq:model-dyn}
    X_{t} &\sim a(\cdot|X_{t-1}) , \quad X_0 \sim \pi_0 \\\label{eq:model-obs}
    Y_t &\sim h(\cdot|X_t)
\end{align}
for $t=1,2,\ldots $ where $X_t\in \mathbb{R}^n$ is the state of the system, $Y_t \in \mathbb{R}^m$ is the observation, $\pi_0$ is the probability distribution for the initial state $X_0$, $a(x'|x)$ is the probability kernel for the transition from the state $x$ to the state $x'$, and $h(y|x)$ is the  likelihood distribution  
of an observation $y$ given a state $x$.
We assume that the update equation~\eqref{eq:model-dyn} is realized with the stochastic map
\begin{equation}\label{eq:dyn}
    X_{t} = \bar  a(X_{t-1},V_t)
\end{equation}
\end{subequations}
where $\{V_t\}_{t=1}^\infty$ is an i.i.d sequence and $\bar a(x,v)$ is Lipschitz in $x$ for all $v$. Throughout the paper, we assume that all probability measures admit a density and use the same notation to refer to
the distribution or the corresponding measure. If needed, the two notions will be distinguished depending on the context. 

The filtering problem is to infer the conditional distribution of the state $X_t$ given the history of the observations $\{Y_1,Y_1,\dots,Y_{t}\}$, that is, the distribution
\[
\pi_t := \mathbb{P}(X_t\in \cdot |Y_1,\dots,Y_{t} ),\quad \text{for}\quad t=1,2,\ldots,
\]
often referred to as the {\it posterior} distribution. 
\subsection{Recursive update for the filter}
The posterior distribution $\pi_t$ admits a recursive update equation that is essential for the design of filtering algorithms. To present this recursive update, let us introduce the following operators:
\begin{subequations}\label{eq:filter}
\begin{align}
    \text{(propagation)}\quad \pi \mapsto \mathcal A \pi &:= \int_{\mathbb R^n} a(\cdot|x) \pi(x)d x\\
  \text{(conditioning)}\quad \pi \mapsto \mathcal B_y \pi &:= \frac{h(y|\cdot)\pi(\cdot)}{\int_{\mathbb R^n} h(y|x) \pi(x)d x} \label{eq:Bayesian}   
\end{align}
The first operator represents the update for the distribution of the state according to the dynamic model~\eqref{eq:model-dyn}. The second operator represents Bayes' rule that
 carries out the conditioning  according to 
the observation model~\eqref{eq:model-obs}.  
In terms of these two operators, the update law
for the posterior is given by~(e.g. see \cite{cappe2009inference}):
\begin{equation}
 \pi_{t} = \calT_t \pi_{t-1}=\mathcal{B}_{Y_t}\mathcal{A} \pi_{t-1}.
\end{equation}
\end{subequations}
where we introduced $\calT_t:=\mathcal{B}_{Y_t}\mathcal{A}$. With slight abuse of notation, we
further
define the  transition operator as
\begin{align*}
  \mathcal  T_{t,s}:= \calT_{t}\circ \dots \circ \calT_{s+1},\quad \forall \quad t> s\geq 0.
\end{align*}
We then have $\pi_t=\mathcal T_{t,s}\pi_s$ for all $t>s\geq 0$. Note that the transition operator $\calT_{t,s}$
is stochastic in nature as  it depends on the realization of the observation signal $\{Y_{s+1},\ldots,Y_t\}$.
We suppress this dependence to simplify the presentation. 

\subsection{Filter stability}
We use the following metric on (possibly random) probability measures $\mu,\nu$:
\begin{equation}\label{eq:metric}
    d(\mu,\nu) := \sup_{g \in \mathcal G} \sqrt{\mathbb{E}\left|\int g d \mu - \int g d \nu\right|^2}
\end{equation}
where  the expectation is over the possible randomness of the probability measures $\mu$ and $\nu$, and $\mathcal G:=\{g:\Re^n \to \Re;\,|g(x)|\leq 1, |g(x)-g(x')|\leq \|x-x'\|,\quad \forall x,x'\}$ is the space of functions that are uniformly bounded by one and uniformly Lipschitz with constant smaller than one (this metric is also known as the dual bounded-Lipschitz distance). 
We use this metric to introduce a notion of uniform geometrical stability for the filter. 
\begin{definition}[Uniformly geometrically  stable filter]
  The filter update~\eqref{eq:filter} 
  is uniformly geometrically stable 
if $\exists \lambda\in (0,1)$ and positive constant $C>0$ such that for all $\mu,\nu$ and $ t>s\geq 0$ it holds that
\begin{equation}\label{eq:stability}
d(\mathcal T_{t,s}\mu,\mathcal T_{t,s}\nu) \leq C(1-\lambda)^{t-s}d(\mu,\nu).
\end{equation}
\label{def:stability}
\end{definition}
\medskip

\begin{remark}\label{rem:stability}
  The uniform geometric stability property~\eqref{eq:stability} is also used in the error analysis of PFs in~\cite{del2001stability,delmoralbook}. It can be verified if the dynamic transition kernel satisfies a minorization condition,
    i.e., there exists a probability measure $\rho$ 
    and a constant $\epsilon >0$ such that  $a(x|x')\geq \epsilon\rho(x)$.
  The minorization is a mixing condition that ensures geometric ergodicity of the Markov process $X_t$~\cite{meyn2012markov}. We acknowledge that this condition is  a strong and can be verified for a restricted class of systems, e.g., $X_t$ should belong to a compact set. A complete characterization of  systems with uniform geometric stable filters is an open and challenging problem in the field. More insight is available for the weaker notion of asymptotic stability of the filter, i.e., $\lim_{t\to\infty}d(\mathcal T_{t,s}\mu,\mathcal T_{t,s}\nu)=0$, which holds when the system is  ``detectable''
  in a sense that is suitable for nonlinear stochastic dynamical systems~\cite{van2010nonlinear,chigansky2009intrinsic,van2009observability,kim2022duality}. This characterization of systems with asymptotic filter stability 
  is in agreement with the existing results for the stability of the Kalman filter, which holds when the linear system is detectable in the classical sense~\cite{ocone1996}.  A complete survey of existing  filter stability results can be found in~\cite{crisan2011oxford}.

\end{remark}

The following Lemma  is useful for our  error analysis. 
\begin{lemma}\label{lem:T-d}
Let $\pi$ be a (random) distribution and $T$ and $S$ two (random) measurable maps. Then,
\begin{align*}
    d(T\#\pi;S\#\pi) \leq \mathbb E \left[\|T-S \|^2_{L^2(\pi)}\right]^{\frac{1}{2}},
\end{align*}
where the expectation is over the possible randomness of the distribution $\pi$ as well as the  maps $T, S$.
\end{lemma}
\begin{proof}
  The proof  follows from a straightforward argument using the definition of the metric~\eqref{eq:metric} 
  and the
  Lipschitz property of the test functions $g$. 
\end{proof}

    \section{Optimal Transport Particle Filters}\label{sec:OTPF}
The construction of OTPFs relies on the variational formulation~\eqref{eq:var-form}. 
Consider the objective functional
\begin{align}\label{eq:obj-func}
J(f;\pi):=\mathbb E[f(\bar X;Y) + f^\star(X;Y)],
\end{align}
where $X \sim \pi$, $Y \sim h(\cdot|X)$, and $\bar X\sim \pi$ is an independent copy of $X$, along with the
optimization problem 
\begin{equation}
    \inf_{f\in \text{CVX}_x}\,J(f,\pi). 
\label{eq:opt-problem}    
\end{equation}
It is shown in~\cite[Prop. 1]{taghvaei2022optimal} that the solution to this  problem provides an OT
characterization of the Bayes operator~\eqref{eq:Bayesian}. The result is reproduced here for completeness
\begin{proposition}\label{prop-1}
Assume $\pi$ admits a density with respect to the Lebesgue measure. Then, the objective function~\eqref{eq:obj-func} has a unique (up to a constant shift) minimizer $\bar f \in \text{CVX}_x$ and
\begin{equation}
     \nabla_x \bar f(\cdot;y) \# \pi = \mathcal B_y \pi ,\quad \text{for a.e.  } y.\label{eq:OT-Bayes}
   \end{equation}
\end{proposition}

\subsection{The exact mean-field process}
We use the OT characterization of the conditional distribution to construct a (exact) mean-field process $\bar X_t$ whose distribution $\bar \pi_t$ is exactly equal to the posterior distribution $\pi_t$.
Consider a process $\bar X_t$ with distribution $\bar \pi_t$  defined as   
\begin{subequations}\label{eq:exact-mean-field}
\begin{equation}\label{eq:Xbar}
  \begin{aligned}
    \Bar{X}_{t} & = \nabla_x \Bar{f}_t(\bar a(\Bar{X}_{t-1},\bar V_t);Y_t) ,\quad \bar X_0 \sim \bar \pi_0 \\
     \bar f_t &= \argmin_{f\in \text{CVX}_x}\,J(f,\mathcal A  \bar \pi_{t-1}),
  \end{aligned}
\end{equation}
 where $\bar V_t$ is an independent copy of $V_t$ in the dynamic model~\eqref{eq:dyn}.
It is then straightforward to verify that
\begin{equation}
\begin{aligned}
  \bar \pi_{t} &=  \nabla_x \bar f_t(\cdot;Y_t) \# \mathcal A  \bar \pi_{t-1}
  = \mathcal B_y \mathcal A \bar \pi_{t-1},
\end{aligned}
\label{eq:exact-mean-field-dist}
\end{equation}
\end{subequations}
where the second identitiy is a consequence of Proposition~\ref{prop-1}. It then follows that
whenever $\bar \pi_0 = \pi_0$ then $\bar \pi_t = \pi_t$. As such, the mean-field process $\bar X_t$ is called exact.
The OTPF is obtained by approximating the exact mean-field process $\bar X_t$ in two steps, as described next.

\subsection{The approximate mean-field process}
The first approximation step consists of restricting the feasible set of the optimization problem~\eqref{eq:opt-problem} to a parameterized class of convex functions $\mathcal F \subset \text{CVX}_x $.
The resulting approximated distribution is denoted by $\pi_t^\calF$ which follows the update rule:
\begin{subequations}\label{eq:approx-mean-field}
\begin{equation}
\begin{aligned}
\pi_{t}^\calF &=  \nabla_x f_t^\calF(\cdot,Y_t)\#\mathcal A\pi_{t-1}^\calF,\quad \pi^\calF_0=\pi_0
\\
 f^\calF_t &= \argmin_{f\in \calF}\,J(f, \mathcal A\pi_{t-1}^\calF)
\end{aligned}
\label{eq:approx-mean-field-dist}
\end{equation}
This update defines the approximate mean-field process 
\begin{align}\label{eq:XF}
    X^\calF_{t} = \nabla_x f^\calF_t(\bar a(X^\calF_{t-1};Y_t),\bar V_t),\quad X^\calF_0 \sim \bar \pi_0. 
\end{align}
\end{subequations}
The approximation error between $\pi^\calF_t$ and $\bar \pi_t$, due to the parameterization of the function $f_t$
is studied in section~\ref{sec:approx-error}. 

\subsection{The finite particle system}

The second approximation step is to replace the mean-field process with an empirical distribution of a collection of
particles $\{X^1_t,\ldots,X^N_t\}$, i.e.,  $\pi^\calF_t \approx \frac{1}{N}\sum_{i=1}^N \delta_{X_t^i}.$
The finite-$N$ discretization can be achieved through two different approaches leading to two different systems of particles. The first system (the particle system with resampling) is more amenable to  error analysis, while the second system
the interacting particle system) is more practical.

\newP{(C.I) the particle system with resampling} Define the sampling operator 
\begin{equation}
    \pi \mapsto \calS^N \pi := \frac{1}{N}\sum_{i=1}^N \delta_{X^i}\quad X^i \overset{\text{i.i.d.}}{\sim} \pi, 
\end{equation}
and approximate the mean-field distribution $\pi_t^\calF$ by introducing the sampling operator $\calS^N$
within the update equations:
\begin{subequations}\label{eq:finite-N}
\begin{equation}\label{eq:particles-dist}
\begin{aligned}
    \tilde \pi^{(\calF,N)}_{t} &=  \nabla_x \tilde f_t^{(\calF,N)}(\cdot,Y_t)\# \calS^N \mathcal A \tilde \pi^{(\calF,N)}_{t-1},\quad\pi^{(\calF,N)}_{0}=\pi_0 \\
     \tilde f^{(\calF,N)}_{t} &= \argmin_{f\in \mathcal{F}}\, J(f,\calS^N\calA\tilde \pi^{(\calF,N)}_{t-1}), 
\end{aligned}
\end{equation}
The presence of the sampling operator ensures that the distribution $\tilde \pi^{(\calF,N)}_{t}$ is an empirical distribution formed by a collection of particles, i.e. 
$\tilde \pi^{(\calF,N)}_{t}= \frac{1}{N}\sum_{i=1}^N \delta_{\tilde X^i_t}.$
Equation \eqref{eq:particles-dist} further identifies an update law for  the particles:
\begin{equation}\label{eq:particles-resampled}
\begin{aligned}
 \tilde  X_{t}^i &= \nabla_x \tilde f_t^{(\calF,N)}(\bar a(\tilde X_{t-1}^{\sigma_i},V^i_t);Y_t)
  \end{aligned}
\end{equation}
\end{subequations}
where $\sigma_i \sim \text{Unif}\{1,2,\ldots,N\}$ and $\{V^i_t\}_{i=1}^N$ are independent copies of $V_t$. The sampling process is similar to the resampling stage in PFs, with the difference being that the weights are uniform in this case. The resampling step makes the particles independent of each other, which significantly  simplifies the error analysis, as seen in Section~\ref{sec:finite-N-analysis}. 

\newP{ (C.II) the interacting particle system} The second approach to constructing the finite-$N$ particle system is to discretize the update equation~\eqref{eq:XF} for the mean-field process $X_t^\calF$ according to
\begin{equation}\label{eq:particles-interacting}
\begin{aligned}
  X_{t}^i &= \nabla_x f_t^{(\calF,N)}(\bar a(X_t^i,V^i_t);Y_t)\\
  f_t^{(\calF,N)} &=\argmin_{f\in \mathcal{F}}\, J(f,\frac{1}{N}\sum_{i=1}^N\delta_{\bar a(X^i_t,V^i_t)}).
  \end{aligned}
\end{equation}
The empirical distribution  $\pi^{(\calF,N)}_t:=\frac{1}{N}\sum_{i=1}^N \delta_{X^i_t}$ does not follow an update-law  similar to the update law for $\tilde \pi_t^{(\calF, N)}$ due to the nature of the operator $\mathcal A$, which smooths
out empirical distributions. Instead, the update for the interacting particle system can be expressed as
\begin{equation*}
  \pi^{(\calF,N)}_{t} =   \nabla f^{(\calF,N)}\# \calA^N \pi^{(\calF,N)}_{t-1}
\end{equation*}
where $\calA^N$ is a stochastic operator that takes any empirical distribution $\frac{1}{N} \sum_{i=1}^N \delta_{x_i}$ and outputs $\frac{1}{N} \sum_{i=1}^N \delta_{a(x^i,V^i)}$. 
Moreover, in contrast to the previous construction (the particle system with resampling), the particles are now correlated, which makes the error analysis challenging (this is often studied under the propagation of chaos analysis~\cite{sznitman1991}) We leave the error analysis of the interacting particle system as the subject of future work. However,we empirically validate the performance of this approximation in Section~\ref{sec:numerics}.  


    \section{Error Analysis}\label{sec:error-analysis}
    The objective of this section is to study the approximation error of the 
OTPFs introduced above.  We begin with the analysis for the approximate mean-field process before
  turning our attention to the particle system with resampling.

\subsection{The mean-field analysis}\label{sec:approx-error}
The distance between the exact mean-field distribution $\bar \pi_t$ and the approximate distribution $\pi_t^\calF$
  is characterized by the following proposition.

\begin{proposition}\label{prop:mean-field}
  Consider $\bar \pi_t$ and $\pi_t^\calF$ as in~\eqref{eq:exact-mean-field}-\eqref{eq:approx-mean-field} respectively. Assume 
\begin{enumerate}
    \item The exact filter is stable according to Definition~\ref{def:stability}. 
    \item There exists $\epsilon_{\calF}>0$ such that  \begin{align}
        \inf_{f \in \cal F}J(f,\calA \pi^{\calF}_t) -    \inf_{f \in \text{CVX}_x}J(f,\calA\pi^{\calF}_t) \leq \epsilon_{\calF},\quad \forall t.\label{eq:rep-error}
    \end{align}
  \item
     For all $y$ and $t$ the function  $f^{\calF}_t( \cdot;y)$ is convex and $\nabla_x f_t^\calF(\cdot;y)$
      is $\beta$-Lipschitz.
    
\end{enumerate}
Then, it holds that
    \begin{align}\label{eq:pitF-bound}
        d(\pi^\calF_t,\pi_t) \leq \frac{C\sqrt{2\beta\epsilon_\calF}}{\lambda},\quad \forall t,
    \end{align}
    with all constants independent of time. 
\end{proposition}

\begin{remark}
The first assumption in the proposition is used to ensure the error produced at each step of the algorithm does not grow with time. The second assumption is related to the representation power of the function class $\calF$ relative to the class of probability distributions introduced by the algorithm $\calA \pi^\calF_t$. For example, this error is zero when $\calF$ is a class of convex and quadratic functions, and the filtering problem is based on a linear Gaussian dynamic and observation model. In this case, probability distributions $\pi_t^\calF$ are Gaussian with the corresponding quadratic optimal function $\bar f_t$. 
In general, it is expected that the error is small when the distributions are inherently simple, e.g. when the problem exhibits low-dimensional structures or regularities. The analysis of these errors is the subject of representation theory~\cite{anthony1999neural,shalev2014understanding}. The last assumption is related to the regularity of the distributions $\mathcal A \pi_t^\calF$ and the resulting posterior distributions and can be  enforced
by an appropriate choice of the class $\calF$.
\end{remark}

\begin{proof}
  To simplify the presentation, we introduce the operator $ \pi \mapsto \calT_t^\calF \pi := \nabla f^\calF_t(\cdot,Y_t)\#\calA\pi$ for all $t$, to denote the update law for the approximate mean-field distribution in~\eqref{eq:approx-mean-field-dist}. The first step in the proof  is to use the triangle inequality and the filter stability to bound the  error between $\pi_t$ and $\pi^\calF_t$
  as follows: 
\begin{equation*}
    \begin{aligned}
    d(\pi_t,\pi_t^\mathcal{F}) & \leq \sum_{k=1}^t d(\calT_{t,k-1}\pi_{k-1}^\mathcal{F},\calT_{t,k}\pi_k^\mathcal{F})\\
    & \leq \sum_{k=1}^t d(\calT_{t,k}\calT_{k}\pi_{k-1}^\mathcal{F},\calT_{t,k}\calT_{k}^\mathcal{F}\pi_{k-1}^\mathcal{F})\\
    & \leq \sum_{k=1}^t C(1-\lambda)^{t-k}d(\calT_{k}\pi_{k-1}^\mathcal{F},\calT_{k}^\mathcal{F}\pi_{k-1}^\mathcal{F})\\
    &\leq \frac{C}{\lambda} \max_{k \in \{1,2,\ldots,t\}}\{d(\calT_{k}\pi_{k-1}^\mathcal{F},\calT_{k}^\mathcal{F}\pi_{k-1}^\mathcal{F})\}.
    \end{aligned}
\end{equation*}
Next, we use Lemma~\ref{lem:T-d} to bound the distance 
\begin{align*}
  d(\calT_{k}\pi_{k-1}^\mathcal{F}&,\calT_{k}^\mathcal{F}\pi_{k-1}^\mathcal{F})\\&=d(\nabla \bar f_k(\cdot;Y_k)\#\calA \pi_{k-1}^\mathcal{F},\nabla f_k^\calF(\cdot;Y_t)\#\calA\pi_{k-1}^\mathcal{F})\\&\leq \mathbb E\left[\|\nabla \bar f_k(\cdot;Y_k) - \nabla f^\calF_k(\cdot;Y_k) \|^2_{L^2(\calA \pi_{k-1}^\calF)}\right]^{\frac{1}{2}} 
\end{align*}
for all $k\geq 0$. 
Finally, we use the second and third assumptions in the proposition to obtain a uniform bound for the error between $\nabla \bar f_k$ and $\nabla f^\calF_k$  using Lemma~\ref{lem:opt-gap} (is outlined below)
\begin{align*}
\mathbb E&\left[\|\nabla \bar f_k(\cdot;Y_k) - \nabla f^\calF_k(\cdot;Y_k) \|^2_{L^2(\calA \pi_{k-1}^\calF)}\right]\\&\leq 2\beta (J(f_k^\calF,\calA\pi_{k-1}^\calF) - J(\bar f_k,\calA\pi_{k-1}^\calF))\\
&\leq 2\beta  \epsilon_\calF
\end{align*}
concluding the final bound~\eqref{eq:pitF-bound}.  
\end{proof}
\begin{lemma}\label{lem:opt-gap}
Consider the optimization problem~\eqref{eq:opt-problem} with the objective function~\eqref{eq:obj-func}. Assume $\pi$ admits density.   Let $\bar f$ be the optimal function and $f$ be an arbitrary convex 
 and $\beta$-smooth function. Then,
\begin{align*}
   J(\pi,f) - J(\pi,\bar f) \geq \frac{1}{2\beta}\mathbb E[ \|\nabla f(\bar X;Y) - \nabla \bar f(\bar X;Y) \|^2]. 
\end{align*}
\end{lemma}
\medskip
\begin{proof}
  The proof is an extension of the result~\cite[Prop. 10]{hutter2019minimax} and omitted on the account of space.
  \end{proof}

\subsection{The particle-system-with-resampling analysis}\label{sec:finite-N-analysis}
Next, we analyze the error between the particle system~\eqref{eq:finite-N} and the exact mean-field process~\eqref{eq:exact-mean-field}. The process is similar to the mean-field analysis presented in the previous section, with an additional error due to the sampling operator and the empirical approximations. 

\begin{proposition}\label{prop:finite-N-resampled}
  Consider the exact mean-field distribution $\bar \pi_t$  and the particle distribution
  $\tilde \pi_t^{(\calF, N)}$
  defined in~\eqref{eq:exact-mean-field} and~\eqref{eq:finite-N}, respectively. Assume
\begin{enumerate}
    \item The exact filter is stable according to Definition~\ref{def:stability}. 
    \item There exists a constant $\epsilon_{\calF,N}>0$ such that  for all  $t$ and $N$:
    \begin{align*}
        \inf_{f \in \calF}\!J(f,\calS^N\!\calA \tilde \pi^{(\calF,N)}_t)\! - \!   \inf_{f \in \text{CVX}_x}\!J(f,\calA\tilde \pi^{(\calF,N)}_t) \leq \epsilon_{\calF,N}
    \end{align*}
  \item For all $y$, $t$, and $N$,  
    the function $ f^{(\calF,N)}_t(\cdot ;y)$ is convex and $\nabla_x f^{(\calF,N)}_t(\cdot ;y)$
    is  $\beta$-Lipschitz.
  \end{enumerate}
Then, it holds that
    \begin{align}\label{empirical-process-error-bound}
      d(\tilde \pi^{(\calF,N)}_t,\pi_t) \leq
      \frac{C}{\lambda} \left(\sqrt{2\beta\epsilon_{\calF,N}} + \frac{1}{\sqrt{N} }\right),\quad \forall t,
    \end{align}
    where all constants are time-independent.    
\end{proposition}


\begin{proof}
  The proof is similar to that of Proposition~\ref{prop:mean-field}.
  Define the operator
  $ \tilde{\calT}_t^{(\calF,N)}: \pi \mapsto \nabla \tilde f^{(\calF,N)}_t(\cdot,Y_t)\#\calS^N \calA\pi$.
  Then, the triangle inequality and  filter stability imply
\begin{align*}
    d(\pi_t,\pi_t^{(\calF,N)}) & \leq  \frac{C}{\lambda} \max_{k \in \{1,2,\ldots,t\}}\{d(\calT_{k}\pi_{k-1}^{(\calF,N)},\tilde \calT_{k}^{(\calF,N)}\pi_{k-1}^{(\calF,N)})\}.
\end{align*}
Applying the  triangle inequality again, we can write
\begin{align*}
d&(\calT_{k}\pi_{k-1}^{(\calF,N)},\tilde \calT_{k}^{(\calF,N)}\pi_{k-1}^{(\calF,N)}) \\&=d( \nabla \bar  
 f_k(\cdot,Y_t)\#\calA \pi^{(\calF,N)}_{k-1},\nabla  f_k^{(\calF,N)}\#\calS^N\calA\pi^{(\calF,N)}_{k-1})\\
&\leq d( \nabla \bar  
 f_k(\cdot,Y_t)\#\calA \pi^{(\calF,N)}_{k-1},\nabla  f_k^{(\calF,N)}\#\calA\pi^{(\calF,N)}_{k-1})\\
&+d( \nabla  f_k^{(\calF,N)}\#\calA\pi^{(\calF,N)}_{k-1},\nabla  f_k^{(\calF,N)}\#\calS^N\calA\pi^{(\calF,N)}_{k-1})
\end{align*}
By application of Lemma~\ref{lem:T-d} and Lemma~\ref{lem:opt-gap}, the first term is upper-bounded by the square-root of  
\begin{align*}
     \mathbb E&\left[\|\nabla \bar f_k(\cdot;Y_k) - \nabla f^{(\calF,N)}_k(\cdot;Y_k) \|^2_{L^2(\calA \tilde \pi_{k-1}^{(\calF,N)})}\right]\\
     &\leq 2\beta(J(f_k^{(\calF,N)},\calA \tilde \pi_{k-1}^{(\calF,N)})-J(\bar f_k,\calA \tilde \pi_{k-1}^{(\calF,N)}))\\
     &\leq 2\beta \epsilon_{\calF,N}
\end{align*}
where we used the second and the third assumptions.
This gives the first term on the right-hand side of \eqref{empirical-process-error-bound}.
The second term 
is due to the sampling error and upper-bounded by $\frac{1}{\sqrt{N}}$ since the test functions $g$ in the definition of the metric $d$ are uniformly bounded by one (e.g. see~\cite[Lemma 2.17]{rebeschini2015can}).  Adding the two errors concludes the final bound. 

\end{proof}

\begin{remark}\label{rem:error}
The assumptions of this proposition are similar to the assumptions in Proposition~\ref{prop:mean-field} with a slight difference in the second assumption. The bound in the second assumption can be decomposed into two terms:
\begin{align*}
        &\inf_{f \in \calF}\!J(f,\calS^N\!\calA \tilde \pi^{(\calF,N)}_t)\! - \!   \inf_{f \in \calF}\!J(f,\calA\tilde \pi^{(\calF,N)}_t) \\
        +&\inf_{f \in \calF}\!J(f,\!\calA \tilde \pi^{(\calF,N)}_t)\! - \!   \inf_{f \in \text{CVX}_x}\!J(f,\calA\tilde \pi^{(\calF,N)}_t) 
    \end{align*}
The second term is similar to the one used in Proposition~\ref{prop:mean-field} and related to the representation power of $\calF$. The first term 
corresponds to the statistical generalization errors due to  approximating distributions with empirical samples and the subject of statistical generalization theory~\cite{shalev2014understanding,arora2017generalization,liu2017approximation,zhang2017discrimination}. The error is expected to scale according to $O(\frac{C_{\mathcal F}}{\sqrt{N}})$  where the constant $C_{\mathcal F}$ is a proxy for the complexity of the class of functions $\mathcal F$, and independent of the dimension $d$. The first term can also be interpreted as the variance, while the second term is the bias. Then our
  error analysis is a manifestation of the bias-variance trade-off dependent on the
  complexity of the function class $\mathcal F$. Similar bias-variance trade-offs also appear in the analysis of local
PFs in~\cite{rebeschini2015can}. 
\end{remark}

        \section{Numerical Experiments}\label{sec:numerics}
     

We use a numerical example to illustrate the proposed \ALG~
in comparison with two other filters: the EnKF~\cite{evensen2006}, and  the
sequential importance resampling (SIR) PF~\cite{doucet09}.

For the \ALG, we solve a  min-max formulation of the variational problem~\eqref{eq:opt-problem}, as described in~\cite[Sec. III-B]{taghvaei2022optimal} and originally proposed in~\cite{makkuva2020optimal} for
estimating OT maps. The min-max formulation involves optimization over an additional convex function 
$\psi$
which is used to represent the convex conjugate $f^\ast$ as follows:
\begin{align*}
    \mathbb E[f^\ast(X;Y)] \!= \!\max_{\psi \in \text{CVX}_x}\!\mathbb E[X^\top \nabla_x \psi(X;Y) \!-\! f(\nabla_x \psi(X;Y);Y)]
\end{align*}
However, in our numerical experiments, we observed that relaxing the constraint and optimizing over a map $T(x;y)$ instead of $\nabla_x \psi(x;y)$ 
produces better numerical results due to the
additional freedom in the parameterization. Therefore, we use the formulation 
\begin{align*}
    \mathbb E[f^\ast(X;Y)] = \max_{T} \mathbb E[X^\top  T(X;Y) \!-\! f(T(X;Y);Y)]
\end{align*}
Note that this does not change the optimization problem because the optimal $T$ is of gradient form and equal to $\nabla_x f^\ast$. The final objective function takes the form
\begin{equation}\label{eq:opt-problem-numerics}
\begin{aligned}
&\min_{f \in \text{ICNN}} \max_{T\in \text{ResNet}} \{ \Expect_{P_{XY}}[f(X,Y)]  \\&\qquad + \Expect_{P_X \otimes P_Y}[X^TT(X,Y) - f(T(X,Y),Y)]\} 
\end{aligned}
\end{equation}
\begin{remark}
Note that we changed the role of source $P_X \otimes P_Y$ and the target $P_{XY}$, compared to the original formulation~\eqref{eq:var-form},   
  so that $T$ represents the transport map from the prior to the posterior, instead of $\nabla_x f$. This formulation leads to a more convenient parameterization of the map. 
  \end{remark}
  
Similar relaxations to the above have also been found to be beneficial for computing Wasserstein barycenters~\cite{fan2020scalable} and Wasserstein gradient flows~\cite{fan2021variational}. 
  Here ICNN denotes the set of  partially input convex neural networks~\cite{amos2016input}.

  To illustrate the performance of the filters, consider the  following dynamics and observation model:
\begin{subequations}\label{eq:model-example}
\begin{align}
    X_{t} &= (1-\alpha) X_{t-1} + 2\sigma V_t,\quad X_0 \sim \mathcal{N}(0,I_n)\\
    Y_t &= h(X_t) + \sigma W_t
\end{align}
\end{subequations}
for $t=1,2,3,\ldots$, where $X_t,Y_t \in \mathbb R^n$, $\{V_t\}_{t=1}^\infty$ and $\{W_t\}_{t=1}^\infty$ are i.i.d sequences of $n$-dimensional standard Gaussian random variables, $\alpha=0.1$ and $\sigma=\sqrt{0.1}$. We use three observation functions:
\begin{align*}
    h(x)=x,\quad h(x)=x \odot x,\quad h(x)=x \odot x \odot x
\end{align*}
where $\odot$ denotes the element-wise (i.e., Hadamard) product when $x$ is a vector.
  
 In order to solve~\eqref{eq:opt-problem-numerics}, we parameterize ICNN as 
\begin{align*}
f(x;y) = \sum_{k=1}^K W_k(x^\top W^x_k+y^\top W^y_k+b_k)_+^2
\end{align*} 
where $W_k\geq 0$, $W^x_k,W^y_k \in \mathbb R^n ~,b_k \in \mathbb R$ for $k=1,\ldots, K$, and $K=32$ is the size of the network. 
The map $T$ is modeled with a standard residual network with two blocks of size $32$ and a ReLU-activation function. 

  We used the  ADAM optimizer to solve the min-max problem with learning rate $10^{-2}$, inner-loop iteration $10$, and the total number of iterations $1024$, which is divided by $2$ after each time step (of the filtering problem) until it reaches $64$. Each iteration involves a random selection of a batch of samples of size $32$ from the total of $N=1000$ particles $\{(X^1_t,Y^1_t),\ldots,(X^N_t,Y^N_t)\}$. Observation samples $Y^i_t$ are produced using the observation model: $Y^i_t\sim h(\cdot|X^i_t)$. Samples from the independent coupling $P_X\otimes P_Y$ are generated by random shuffling.  The number of particles $N$ is the same for all algorithms. The details of the numerical code is available online\footnote{\url{https://github.com/Mohd9485/OT-EnKF-SIR}}.


The numerical results are presented in Figure~\ref{fig:true_states} for a two-dimensional problem $n=2$, while the figure only shows the first component (We choose $n=2$ because the SIR and OT approach did not differ significantly when $n=1$, while the difference became apparent with $n=2$.) The figure shows the trajectory of the particles along with the trajectory of the hidden state. The first experiment, depicted in panel (a), illustrates the performance of a linear observation function. As expected, all three algorithms behave similarly as all of them are able to capture the exact solution, which is Gaussian in this case, and obtained using linear maps. 
The second experiment, depicted in panel (b), involves the quadratic observation function $h(x) = x \odot x$.
This is an interesting case since the problem is not observable, and we expect to see a (symmetric) bimodal distribution. It is observed that EnKF fails to represent the bimodal distribution while
both OT and SIR capture the two modes, although, SIR exhibits  mode collapse in the time range of $t\in [2, 3.5]$.
Finally, both SIR and OT perform better than ENKF for the cubic observation function
$h(x) = x \odot x \odot x$, depicted in panel (c), as expected due to the strong nonlinearity in
the observation model.

We also quantify the performance of all algorithms in these three experiments by computing the mean-squared-error (MSE) in estimating a function $\phi$ of the state:
\begin{equation}\label{eq:performance_test}
        \text{MSE}_t(\phi) = \mathbb{E} \|\frac{1}{N}\sum_{i=1}^N\phi(X_t^i)-\phi({X}_t)\|^2 
\end{equation}  
where the empirical average approximates the expectation over $100$ independent simulations.  The results
  are depicted 
  in Figure~\ref{fig:mse}.  For the linear and cubic 
  observation models, we used $\phi(x)=x$.
For the quadratic case, we used $\phi(x) = \max(0,x)$
(comparing the estimated and  true means is not a good criterion for the quadratic case because the distribution is bimodal with mean equals to zero).
\begin{figure*}[t]
    \centering
     \subfigure[]{
         \centering
         \includegraphics[width=0.3\hsize]{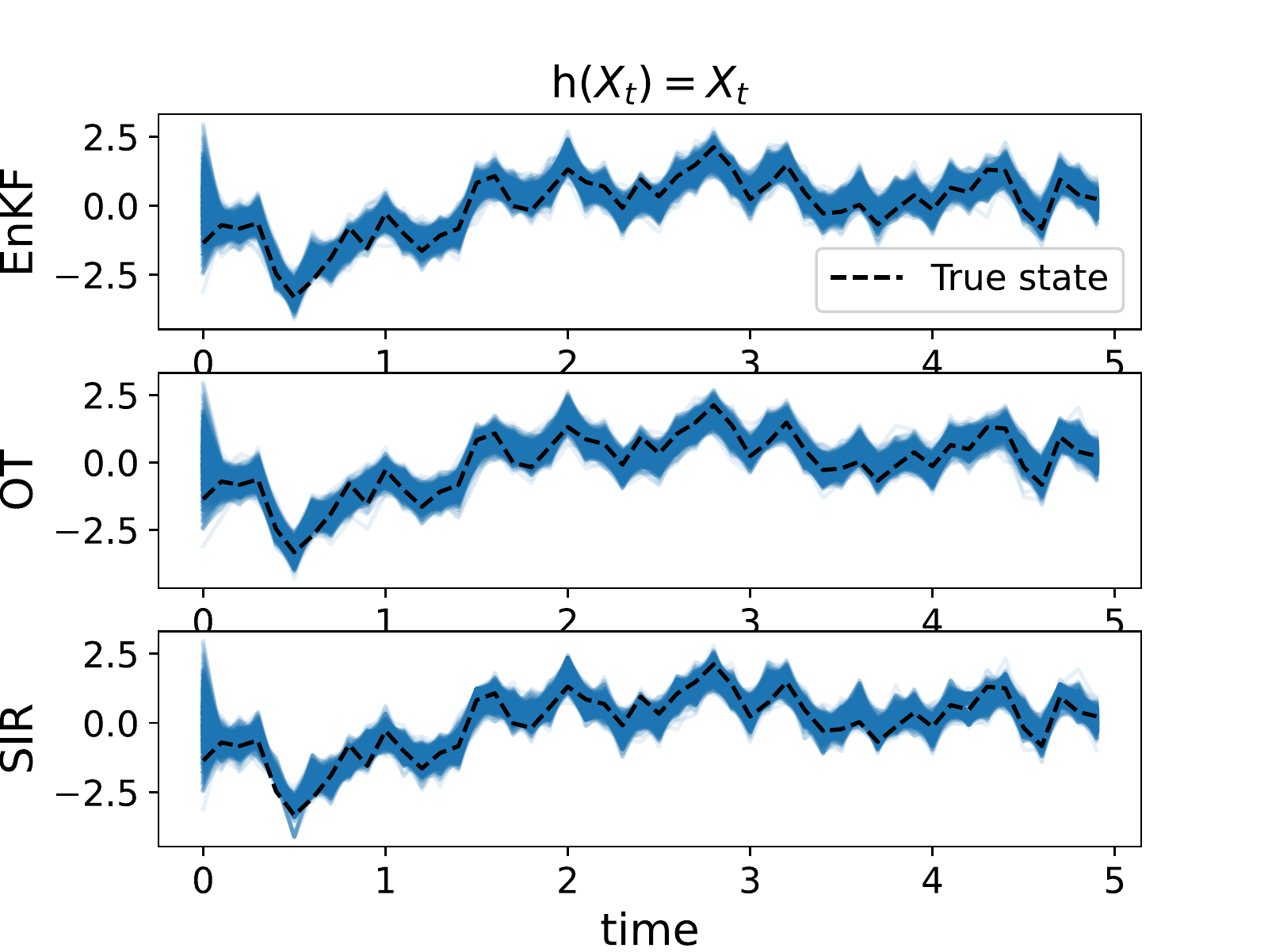}
     }
     \subfigure[]{
     \centering
         \includegraphics[width=0.3\hsize]{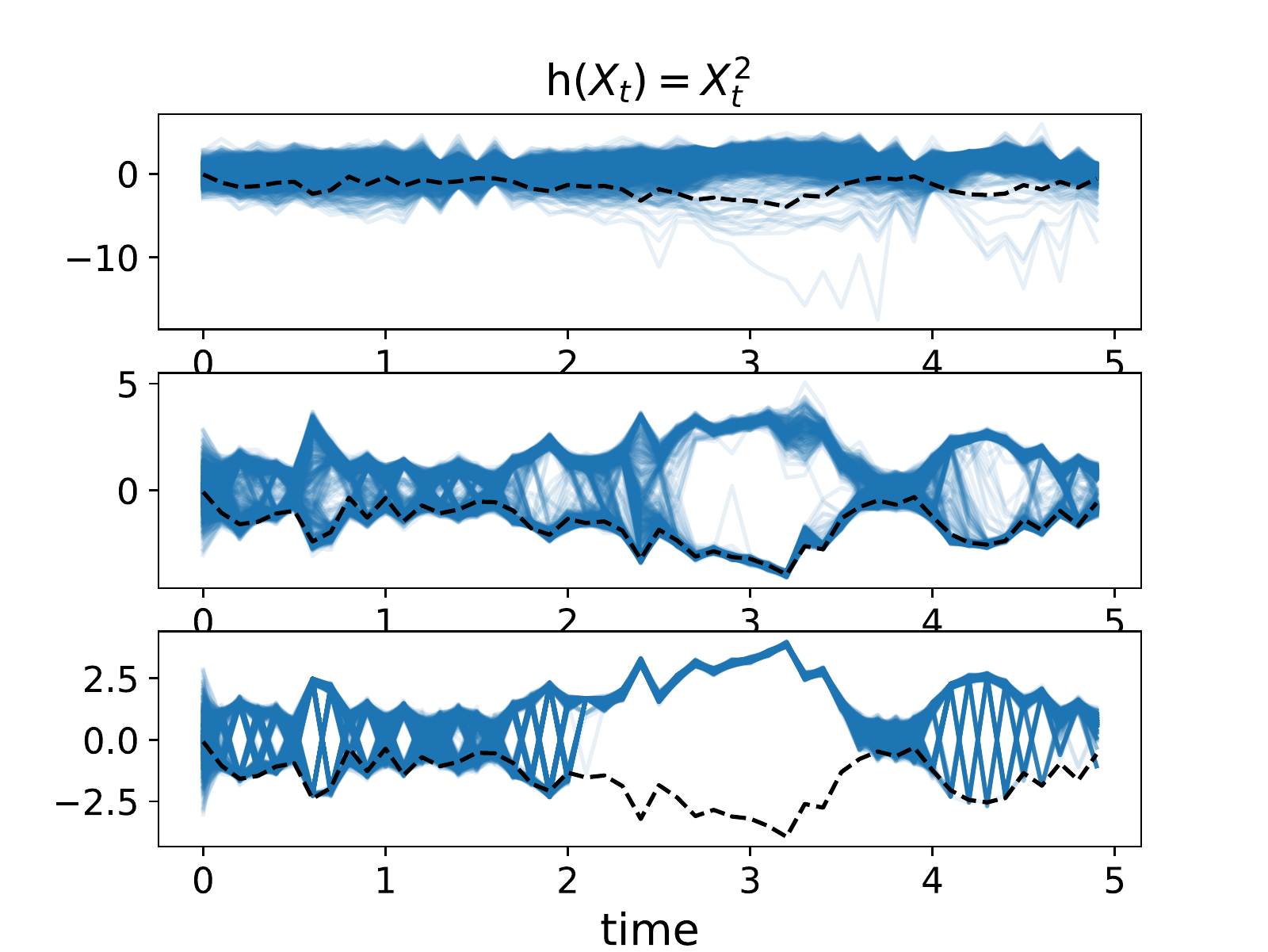}
     }
     \subfigure[]{
     \centering
         \includegraphics[width=0.3\hsize]{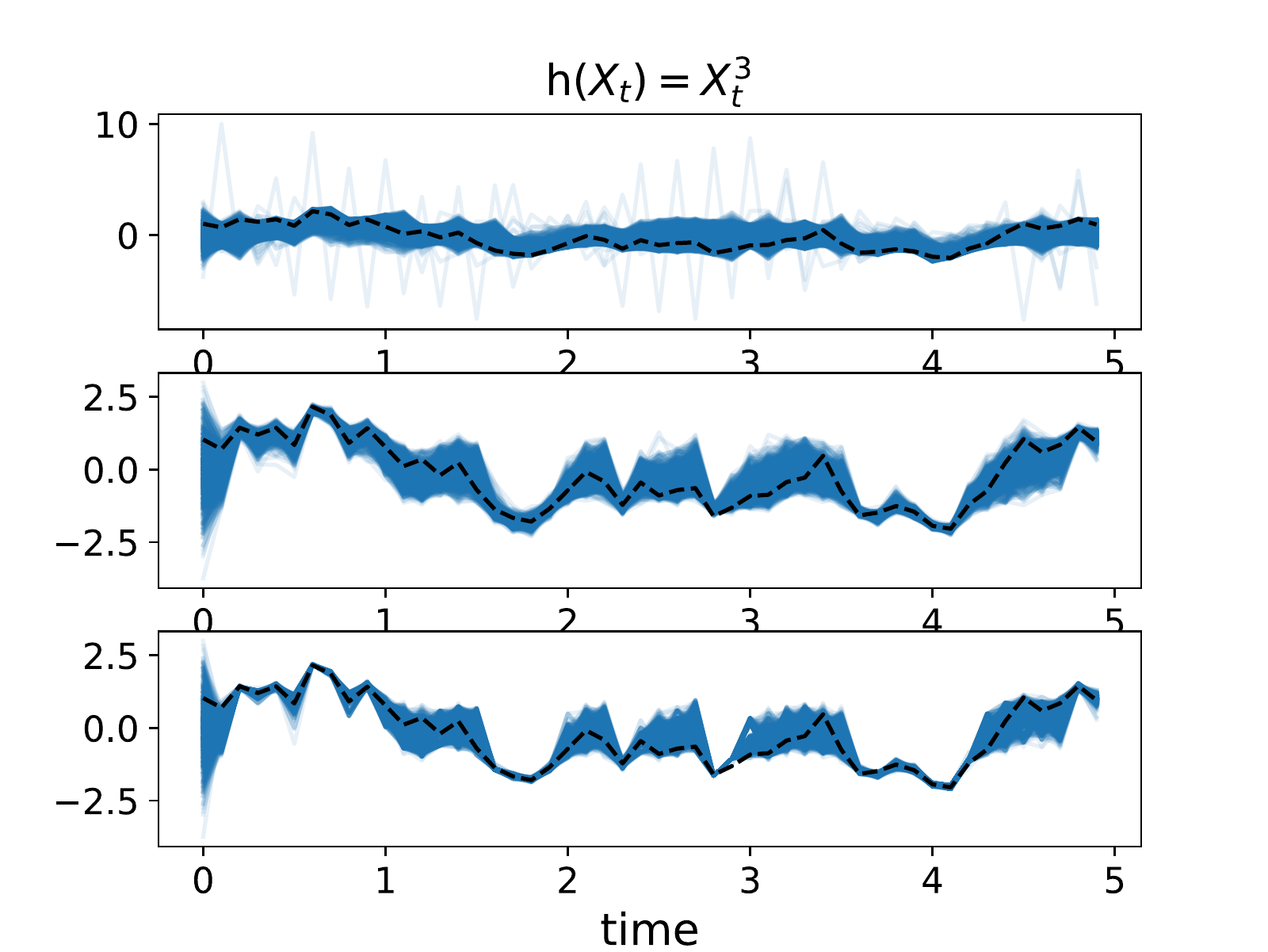}
     }
     \caption{Numerical result for the application of the ensemble Kalman filter, optimal transport particle filter, and sequential importance resampling particle filter, denoted by EnKf, OT, and SIR in the figure respectively, on the numerical example~\eqref{eq:model-example}. The figure shows the trajectory of the particles $\{X^1_t,\ldots,X^N_t\}$ along with the trajectory of the true state $X_t$. The result include three observation functions: (a) $h(x)=x$, (b) $h(x)=x\odot x$, and (c) $h(x)=x\odot x\odot x$.}
    \label{fig:true_states}
\end{figure*}
\begin{figure*}[t]
    \centering
     \subfigure[]{
         \centering
         \includegraphics[width=0.3\hsize]{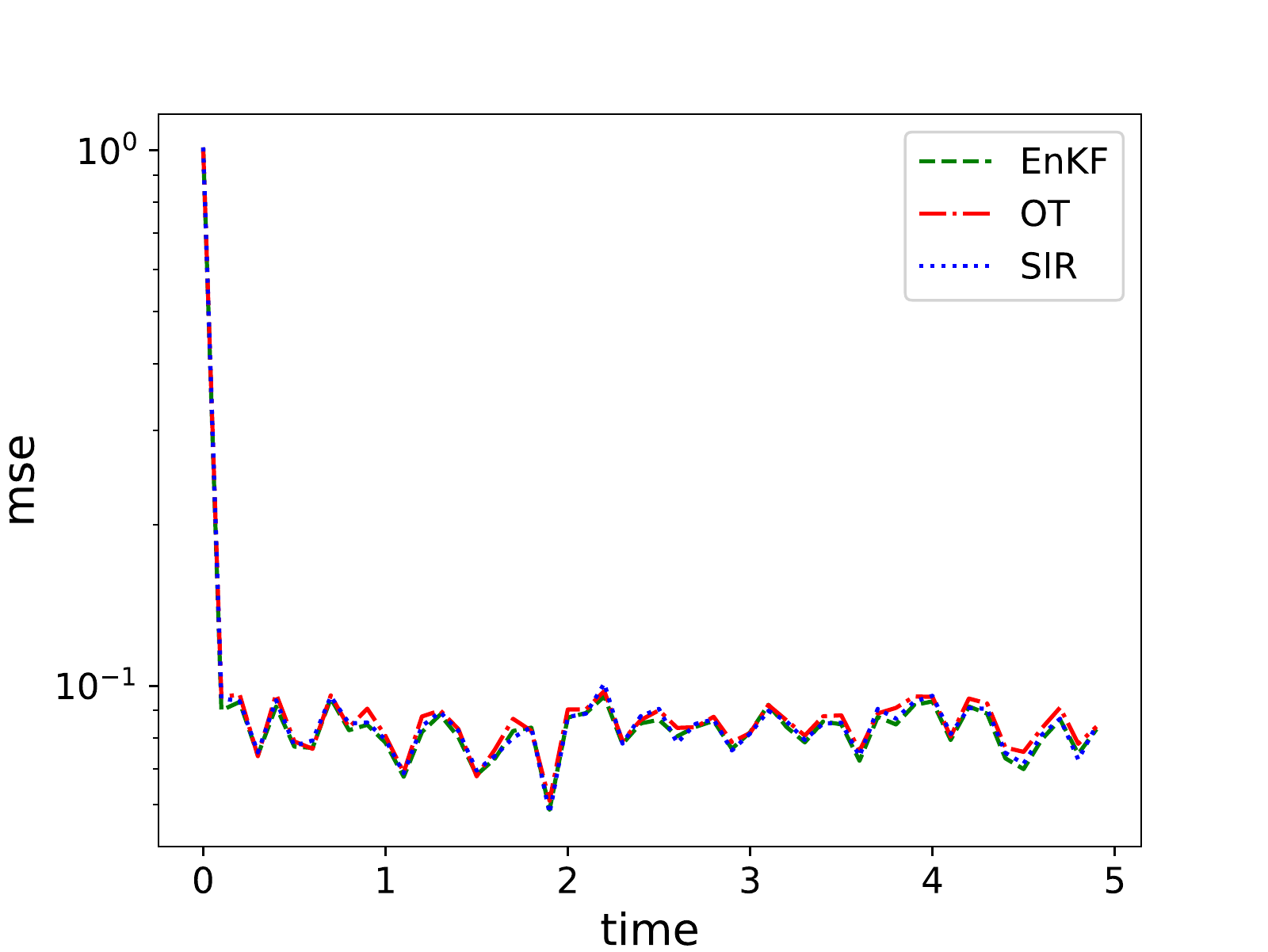}
     }
     \subfigure[]{
     \centering
         \includegraphics[width=0.3\hsize]{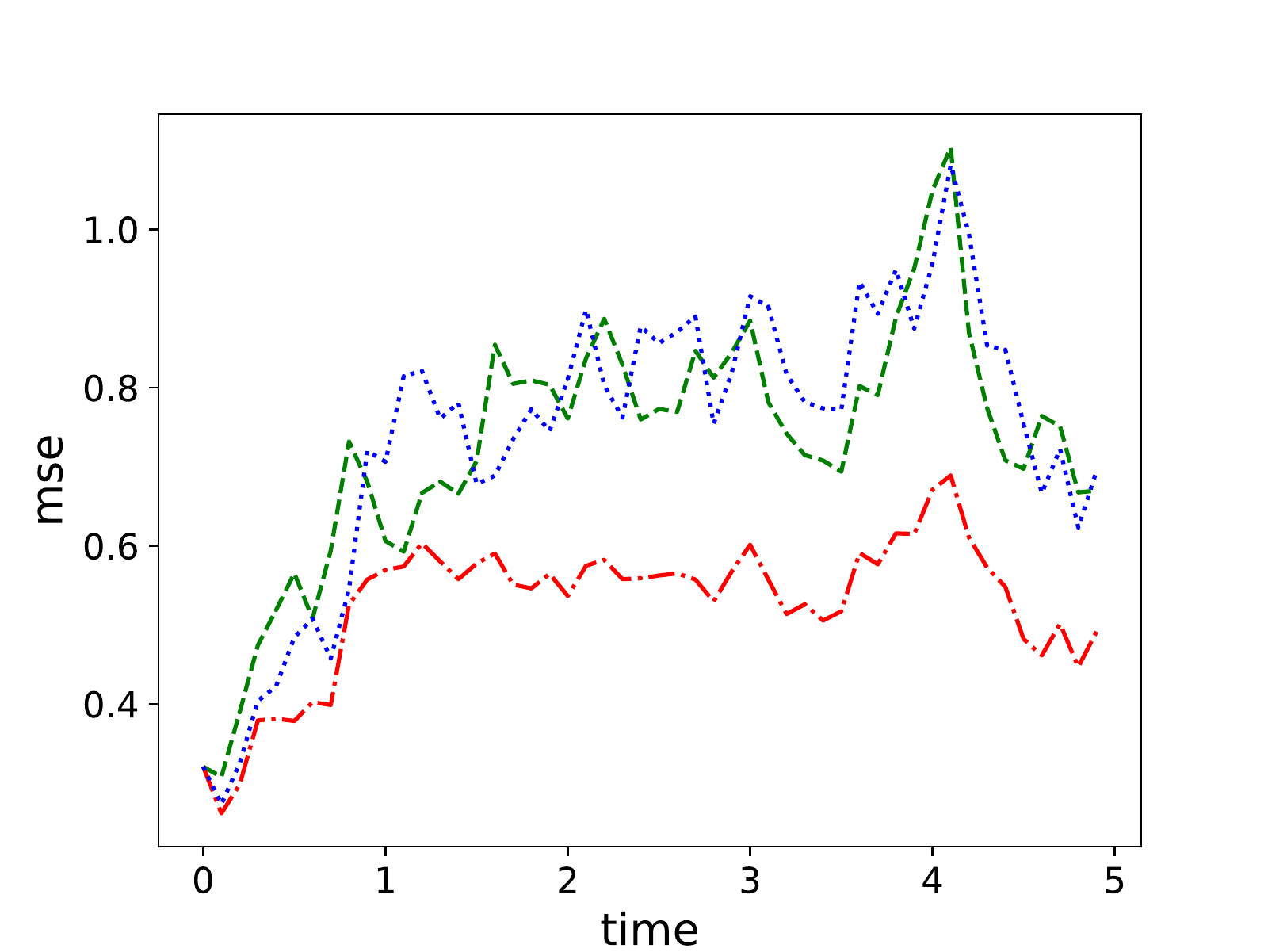}
     }
     \subfigure[]{
     \centering
         \includegraphics[width=0.3\hsize]{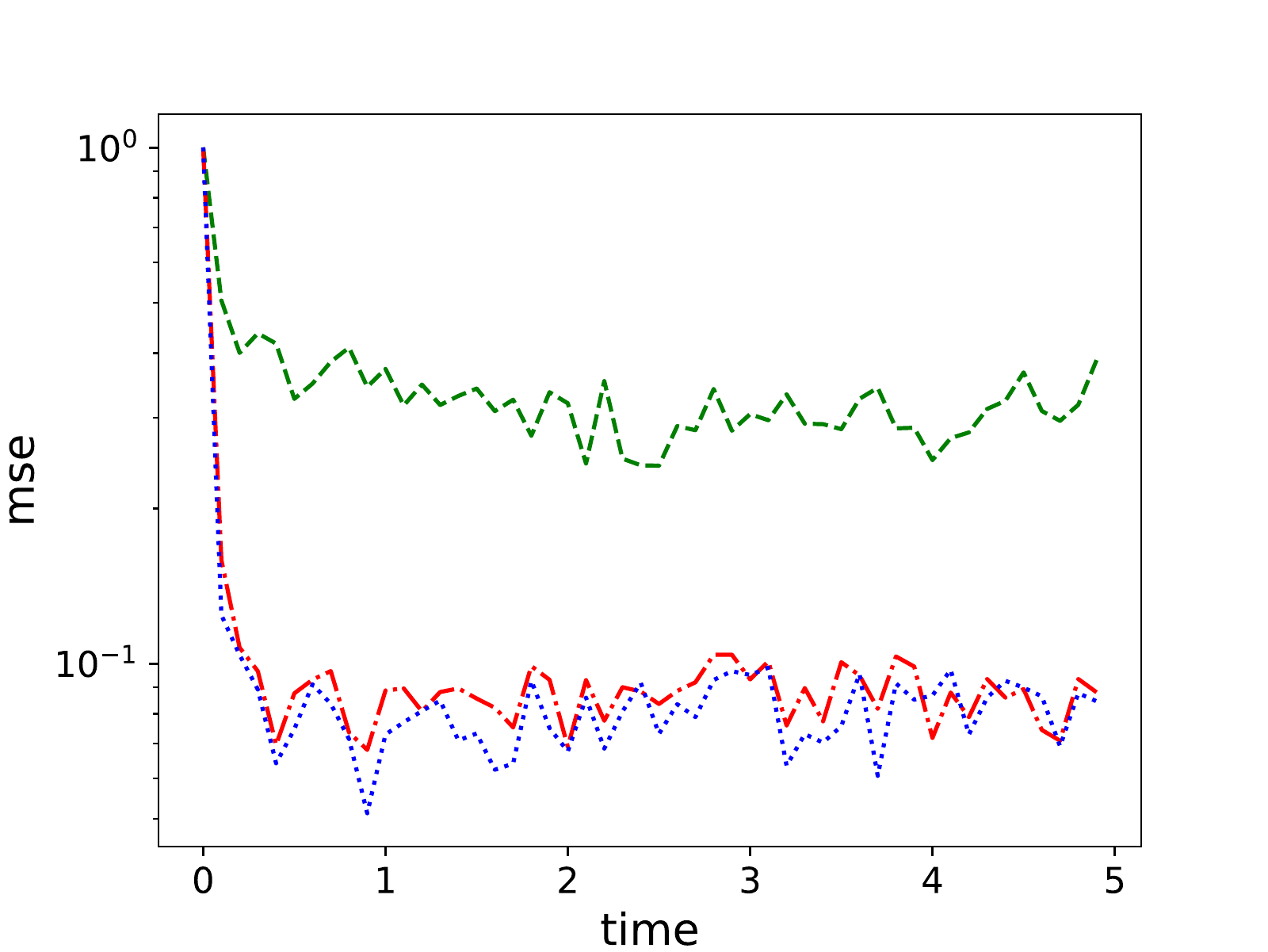}
     }
     \caption{
     Numerical results for the application of filters to the example~\eqref{eq:model-example} in a similar setting as Figure~\ref{fig:true_states}. The figure shows the MSE~\eqref{eq:performance_test} in estimating a function of the state. In panels (a) and (c), the function $\phi(x)=x$ is used, while panel (b)  is for $\phi(x)=\max(0,x)$. The MSE is evaluated by taking the empirical average over $100$ independent simulations.
     }
    \label{fig:mse}
\end{figure*}

In Figure~\ref{fig:mse}-(a), it is observed that both OT and SIR filters yield results that are close 
to the EnKF, which is asymptotically exact for the linear case. However, in Figure~\ref{fig:mse}-(b), the OT method outperforms both EnKF and SIR for the quadratic case, while the difference between OT and SIR is not significant for the cubic case, depicted in Figure~\ref{fig:mse}-(c). 
The performance of the OT filter is expected to improve with further fine-tuning, increasing the iteration number of training, 
and the number of parameters in the neural net, 
at the cost of  higher computational effort.
An appropriate  analysis of the efficiency of the OT method, how it  scales to high-dimensional problems, and its
  application to more realistic data, is the subject of future work.
    
    \section{Discussion}\label{sec:discussion}
In this paper, we presented the OTPF algorithm and provided preliminary theoretical error analysis and numerical results
{that demonstrated the competitive performance of our method in the presence of nonlinear observations and
non-Gaussian states.
We introduced several directions of future research:} the verification of the geometric stability for dynamical systems  e.g. of the form~\eqref{eq:model-example}; error analysis of the optimization gap  in solving the variational problem, both in terms of representation and generalization as discussed in Remark~\ref{rem:error}; error analysis of the interacting particle system without resampling; and extensive numerical experiments and comparison in truly high-dimensional settings. 

    \bibliographystyle{plain}
    \bibliography{TAC-OPT-FPF,references}

\end{document}